\documentclass[reqno]{amsart}
\usepackage[bookmarks=false]{hyperref} 
\usepackage{amssymb}
\usepackage{amsmath}
\usepackage{amsthm} 
\usepackage{mathtools}
\usepackage{color} 
\usepackage[svgnames]{xcolor}
\usepackage[utf8]{inputenc}
\usepackage[T1]{fontenc}
\usepackage{graphicx}
\usepackage{placeins}
\usepackage{vmargin}
\usepackage{setspace}

\hypersetup{
  colorlinks=true,
  linkcolor=Blue  ,          
  citecolor=Red,        
  filecolor=Magenta,      
  urlcolor=Green,           
pdfpagemode=UseOutlines,
pdftitle={},
pdfauthor={Tin Van Phan <van-tin.phan@univ-tlse3.fr>},
pdfsubject={Infinite soliton for derivative nonlinear Schr\"odinger equation},
pdfkeywords={soliton, Schr\"oginer equation} 
}

\numberwithin{equation}{section}
\numberwithin{figure}{section}

\newtheorem{theorem}{Theorem}[section]
\newtheorem{proposition}[theorem]{Proposition}
\newtheorem{lemma}[theorem]{Lemma}

\newtheorem*{assumption a}{Assumption A}
\newtheorem*{assumption b}{Assumption B}
\newtheorem*{assumption c}{Assumption C}

\theoremstyle{definition}

\theoremstyle{remark}
\newtheorem{remark}[theorem]{Remark}

\DeclarePairedDelimiter{\norm}{\lVert}{\rVert}

\newcommand{\N}{\mathbb{N}}
\newcommand{\R}{\mathbb{R}}
\newcommand{\C}{\mathbb C}

\renewcommand{\leq}{\leqslant}
\renewcommand{\geq}{\geqslant}

\DeclareMathAlphabet{\mathpzc}{OT1}{pzc}{m}{it}
\renewcommand{\Re}{\mathcal R\!\mathpzc{e}}
\renewcommand{\Im}{\mathcal I\!\mathpzc{m}}

\begin{document}

\title[On GDNLS]{On a generalized derivative nonlinear Schr\"odinger equation}

\author[Phan Van Tin]{Phan Van Tin}
\address[Phan Van Tin]{LAGA (UMR 7539),
\newline\indent
Institut Galil\'ee, Universit\'e Sorbonne Paris Nord,
\newline\indent
  99 avenue Jean-Baptiste Cl\'ement,
  \newline\indent
  93430 Villetaneuse,
  France}
\email[Phan Van Tin]{vantin.phan@math.univ-paris13.fr}

\subjclass[]{35Q55}

\date{\today}
\keywords{Derivative nonlinear Schr\"odinger equations, scattering, wave operator}

\begin{abstract}
We consider a generalized derivative nonlinear Schr\"odinger equation. We prove existence of wave operator under an explicit smallness of the given asymptotic states. Our method bases on studying the associated system used in \cite{Tinpaper4}. Moreover, we show that if the initial data is small enough in $H^2(\R)$ then the associated solution scatters up to a Gauge transformation in sense of Theorem \ref{thm2}. 
\end{abstract}

\maketitle

\tableofcontents

\section{Introduction}

We consider the following equation
\begin{equation}\label{eq1}
iu_t+u_{xx}+i|u|^{2\sigma}u_x=0,
\end{equation}
for given $\sigma>0$, $u:\R_t\times\R_x\rightarrow\C$ is unknown function and $u_x$ denotes the derivative in space of the function $u$.\\

Local well posedness of \eqref{eq1} was studied in \cite{HaOz16} on general domain. More precisely, for $\sigma\geq 1$, the authors proved that \eqref{eq1} is local well posedness in $H^1$ and global in time if initial data is small enough. Furthermore, in the case $0<\sigma<1$, there exists a solution in $(C_{w}\cap L^{\infty})(\R,H^1)$ for given initial data in $H^1$. \\

In the special case $\sigma=1$, \eqref{eq1} is known to be completely integrable. Global well posedness was studied in \cite{Wu15}. The author proved that for each initial data in $H^1(\R)$ with the mass bounded by $4\pi$, the associated solution is global and uniformly bounded in $H^1(\R)$. Later, in \cite{FuHaIn17}, the authors improved this result by using variational argument. Moreover, the authors showed that the solutions of \eqref{eq1} is global if its mass equals $4\pi$ and its momentum is negative. Using integrability of \eqref{eq1}, solutions of \eqref{eq1} were proved global in $H^{2,2}(\R)$ by \cite{JeLiPeSu18} and in a subset of $H^2(\R)\cap H^{1,1}(\R)$ by \cite{PeSh18}. To our knowledge, the best result is given in \cite{BaPe22} showing that solutions to \eqref{eq1} are global in $H^{\frac{1}{2}}(\R)$ and its $H^{\frac{1}{2}}(\R)$ norm remains globally bounded in time. \\

The equation \eqref{eq1} admits two parameters family of solitary waves solutions (or solitons) $u_{\omega,c}$ for $c\in (-2\sqrt{\omega},2\sqrt{\omega})$. Stability and instability of such solutions were studied in many works (see for instance \cite{MiTaXu23}, \cite{Gu18}, \cite{LiSiSu13}, \cite{GuNiWu20}, \cite{Fu17}, \cite{CoOh06}, \cite{Hayashi22}, \cite{KwWu18}, \cite{NiOhWu17}, \cite{Ning20}, \cite{Kim24} and references therein). It turns out that stability and instability of solitons depend on parameters $\sigma,\omega,c$. More precisely, solitons are unstable if $\sigma\geq 2$, stable if $0<\sigma\leq 1$. When $\sigma\in (1,2)$, there exists $z_0\in (-1,1)$ such that solitons are stable if $-2\sqrt{\omega}<c<2z_0\sqrt{\omega}$ and unstable if $2z_0\sqrt{\omega}\leq c<2\sqrt{\omega}$. There are special solutions, called \textit{multi-solitons}, of \eqref{eq1} which behave in large time like sum of finite solitons \cite{Tinpaper3}, \cite{Tinpaper4}. Stability of such solutions was studied in \cite{CoWu18} when $\sigma=1$ and \cite{TaXu18} when $\sigma\in (1,2)$. More precisely, the authors proved that sum of finite numbers of stable solitons is stable in some sense.  \\

Scattering theory of \eqref{eq1} was first studied in \cite{BaWuXu20}. The authors showed that when $\sigma\geq 2$, solutions are global and scattering when the initial data is small in $H^s$ for $\frac{1}{2}\leq s\leq 1$. Moreover, when $0<\sigma<2$, there exists a class of solitons which are close arbitrary zero, this is against the small data scattering statement. Relating to scattering theory, in \cite{Oz96,HaOz1994}, the authors proved existence of modified wave operator when $\sigma=1$ under a small condition on the asymptotic state given at infinity. Moreover, the modified scattering operator for derivative nonlinear Schr\"odinger equation was proved in \cite{GuHaLiNa13}. Recently, in \cite{baishen23}, the authors proved existence of wave operator by given asymptotic state in $H^1$ for all $\sigma\in\N$, $\sigma\geq 3$. The authors used a Gauge transformation and the normal form method to deal with the appearance of the derivative term in the nonlinearity. In this paper, we prove existence of wave operator under an explicit smallness of given asymptotic state by another approach, specially, we remove the assumption $\sigma\in\N$ used in \cite{baishen23}. \\

The equation \eqref{eq1} admits three conserved quantities:
\begin{align*}
\text{(Energy)}&\quad E(u)=\frac{1}{2}\norm{u_x}^2_{L^2}-\frac{1}{2\sigma+2}\Re\int_{\R} i|u|^{2\sigma}\overline{u}u_x dx,\\
\text{(Mass)}&\quad M(u)=\norm{u}^2_{L^2},\\
\text{(Momentum)}&\quad P(u)=\Re\int_{\R}iu_x\overline{u}dx.
\end{align*}
Let $(\omega,c) \in \R^+\times\R$ be such that
\begin{equation}
\label{eq:condition on parameters}
-2\sqrt{\omega}<c< 2\sqrt{\omega}.
\end{equation}
Define, for all $\varphi\in H^1(\R)$:
\begin{align*}
S_{\omega,c}(\varphi)&=E(\varphi)+\frac{\omega}{2}M(\varphi)+\frac{c}{2}P(\varphi)\\
K_{\omega,c}(\varphi)&=\norm{\varphi_x}^2_{L^2}+\omega\norm{\varphi}^2_{L^2}+cP(\varphi)-\Re\int_{\R}i|\varphi|^{2\sigma}\overline{\varphi}\varphi_x dx,\\
\mu(\omega,c)&=\inf\left\{S_{\omega,c}(\varphi):\varphi\in H^1(\R)\setminus\{0\}, K_{\omega,c}(\varphi)= 0 \right\},\\
\mathcal{K}_{\omega,c}&=\left\{\varphi\in H^1(\R): S_{\omega,c}(\varphi)\leq \mu(\omega,c), K_{\omega,c}(\varphi)\geq 0\right\},\\
\mathcal{K}&=\underset{\substack{-2\sqrt{\omega}<c<2\sqrt{\omega} \\ \omega>0}}{\bigcup} \mathcal{K}_{\omega,c}.
\end{align*}
In \cite{FuHaIn17}, the authors proved that under sufficient conditions, solutions of \eqref{eq1} are global. More precisely, we have the following result:
\begin{theorem}[\cite{FuHaIn17}]
\label{thm:global}
Let $\sigma\geq 1$ and $(\omega,c)$ satisfy \eqref{eq:condition on parameters}. If the initial data $u_0\in \mathcal{K}_{\omega,c}$ then the associated solution of \eqref{eq1} is global and its $H^1$-norm is uniformly bounded in time. Specially, if $u_0\in\mathcal{K}$ then the associated solution of \eqref{eq1} is global.
\end{theorem}

In \cite{HaOz94,Ha93,HaOz92}, the authors used a Gauge transformation to construct solutions to derivative nonlinear Schr\"odinger equations by given initial data. Using these Gauge transformations, we show that there exists a solution scattering to a given asymptotic state. This solution is well defined in large time. To show that the wave operator is well defined, we need to prove that this solution can extend at least to zero and exists uniquely in some sense. Using Theorem \ref{thm:global}, under suitable conditions of asymptotic state, the solution scattering to this state exists globally in time. \\

Our first main result is the following.
\begin{theorem}\label{thm1}
Let $\sigma> 2$ and $u^+ \in H^3(\R) \cap W^{2,1}(\R)$ be a given function. There exists a solution $u$ of \eqref{eq1} scatters forward to $u^+$. The convergence is in polynomial rate, more precisely, for some $T_0\gg 1$:
\begin{equation}
\label{eq:rate of scattering}
\norm{u-e^{it\partial_x^2}u^+}_{H^1}\lesssim t^{1-\sigma},\quad \forall t\geq T_0.
\end{equation}
Assume that $u^+$ satisfies 
\begin{equation}
\label{eq:condition on u^+}
\norm{\partial_x u^+}^2_{L^2}+\omega \norm{u^+}^2_{L^2}<2\mu(\omega,0),
\end{equation}
for some $\omega>0$ then the associated solution $u$ of \eqref{eq1} is global. Furthermore, if $v\in C(\R,H^1)\cap L^4_{\text{loc}}(\R,W^{1,\infty})$ is a solution of \eqref{eq1} such that $\norm{v_x}_{L^{\infty}_x(\R)}$ is uniformly bounded on $[T_0,\infty)$ and $v$ scatters to $u^+$ in the sense of \eqref{eq:rate of scattering} then $u(t)=v(t)$ for all $t\in\R$. In this case, the wave map: $u^+ \mapsto u_0=u(0)$ is well defined.  
\end{theorem}

\begin{remark}
If $K_{\omega,0}(\varphi)=0$ then $$S_{\omega,0}(\varphi)=\left(\frac{1}{2}-\frac{1}{2\sigma+2}\right)\int (|\partial_x\varphi|^2+\omega |\varphi|^2) dx=\frac{\sigma}{2\sigma+2}\int (|\partial_x\varphi|^2+\omega |\varphi|^2) dx.$$
As in \cite{FuHaIn17}, we have 
$$\mu(\omega,0)=S_{\omega,0}(\phi_{\omega,0})=\frac{\sigma}{2\sigma+2}\int (|\partial_x\phi_{\omega,0}|^2+\omega|\phi_{\omega,0}|^2) dx>0,$$
where 
$$\phi_{\omega,0}^{2\sigma}=2\sqrt{\omega}(\sigma+1)\cosh^{-1}(2\sqrt{\omega}\sigma x).$$
Thus, the condition \eqref{eq:condition on u^+} is equivalent to 
$$\norm{\partial_x u^+}^2_{L^2}+\omega\norm{u^+}^2_{L^2}<\frac{2\sigma}{2\sigma+2}\int (|\partial_x\phi_{\omega,0}|^2+\omega|\phi_{\omega,0}|^2)dx.$$
\end{remark}

\begin{remark}
The restriction on $\sigma$ in the above theorem (i.e $\sigma>2$) is required to ensure that $s^{1-\sigma} \in L^1(1,\infty)$. This implies that the right hand side of \eqref{eq 99} tends to zero as $N$ tends to infinity. Thus, we obtain the desired relation for the solution given in Step 1 of the proof of Theorem \ref{thm1}, which plays an important role in the proof.  
\end{remark}

\begin{remark}
The restriction on regularity of the final state in Theorem \ref{thm1} make weaker our result. In \cite{baishen23}, the authors only use condition $u^+ \in H^1(\R)$.
\end{remark}

For $u\in H^1$, we define the Gauge transformation of $u$ by
\begin{align*}
G_1(u)&=\exp\left(\frac{i}{2}\int_{-\infty}^x|u(t,y)|^{2\sigma}\,dy\right)u(t,x),\\
G_2(u)&=\exp\left(\frac{i}{2}\int_{-\infty}^x|u(t,y)|^{2\sigma}\,dy\right)\partial_x u(t,x).
\end{align*}

In the case when the initial data is small in $H^2(\R)$, we show that the associated solution scatters up to the Gauge transformation in the sense of Theorem \ref{thm2}. Our second result is the following.
\begin{theorem}\label{thm2}
Let $\sigma\geq 3$, $u_0\in H^2$ be such that $\norm{u_0}_{H^2}$ is small enough and $u$ be the associated solution of \eqref{eq1} with initial data $u_0$. Then there exists a unique $\varphi^{\pm},\psi^{\pm} \in H^1$ such that
\[
\norm{G_1(u)-e^{it\partial_x^2}\varphi^{\pm}}_{H^1}+\norm{G_2(u)-e^{it\partial_x^2}\psi^{\pm}}_{H^1}\rightarrow 0,
\]
as $t\rightarrow\pm\infty$.
\end{theorem}

\begin{remark}
The assumption $\sigma\geq 3$ in the above theorem is a technical condition, which is useful in the proof of this theorem. Our result in Theorem \ref{thm2} is not the same in \cite{BaWuXu20} where the authors showed that the solutions of \eqref{eq1} scatters under the smallness of $u_0$ in $H^s$ for $\frac{1}{2}\leq s\leq 1$. However, from \cite{BaWuXu20}, we conjecture that the above theorem holds true when $\sigma\geq 2$.   
\end{remark}

\begin{remark}
It is well known that if initial data $u_0$ is in $H^2$ then the associated solution $u$ of \eqref{eq1} is defined uniquely (see for instance, \cite[Theorem 1.1]{HaOz16}). From theorem \ref{thm2}, the solution is global in time. We refer the reader to \cite{Wu15,FuHaIn17,PeSh18,JeLiPeSu18,BaPe22} and conferences therein for the other results on global existence of solutions of \eqref{eq1}.  
\end{remark}

\begin{remark}
Let $u\in H^{s+1}$ for some $s>\frac{1}{2}$. Then $(\varphi,\psi)=(G_1(u),G_2(u)) \in H^{s+1}\times H^s$. Moreover, $\norm{\varphi}_{H^{s+1}}+\norm{\psi}_{H^s}\lesssim C(\norm{u}_{H^s})$, where $C: \R^+ \rightarrow \R^+$ is a continuous, locally bounded function. Thus, we can also work in $H^{s+1}$, $s>\frac{1}{2}$ and obtain a similar result of Theorem \ref{thm2}. However, to avoid the complexities from fractional derivatives, we only work in $H^2$ as in Theorem \ref{thm2}.
\end{remark}

Our strategy to prove Theorem \ref{thm1} is to use the Gauge transformations of $u$ to obtain the system \eqref{system new} of two Schr\"odinger equations without derivative nonlinearities. In \cite[Lemma 3.8]{Tinpaper4}, we studied this system to prove existence of multi-solitons, where the profile $H$ decays exponentially in time. Our method used in \cite{Tinpaper4} is inspired by \cite{CoDoTs15,CoTs14}, where the authors proved existence of multi-solitons for classical nonlinear Schrödinger equations. In our case, the profile $H$ only decays polynomially in time. Combining to the fact that the profile $W$ decays polynomially in time, it suffices to show that there exists a unique solution of \eqref{system new} which decays polynomially in time and hence we obtain a desired solution of \eqref{eq1} scattering to the given asymptotic profile by polynomial rate. One of difficulty in the proof is to prove a relation of $\varphi,\psi$, the Gauge transformation of $u$. More precisely, we need to prove $\psi=\partial_x\varphi-\frac{i}{2}|\varphi|^{2\sigma}\varphi$. Define $\kappa=\partial_x\varphi-\frac{i}{2}|\varphi|^{2\sigma}\varphi$, hence, it suffices to show that $\psi=\kappa$. Instead of considering the quantity $\norm{\psi(t)-\kappa(t)}_{L^2}^2$ as in \cite{HaOz94}, we consider the quantity $\norm{\tilde{\psi}(t)-\tilde{\kappa}(t)}_{L^2}^2$ since we already have the boundedness of $\tilde{\psi}$ and $\tilde{\kappa}$ in large time by Proposition \ref{pro exist solution}.\\

Theorem \ref{thm2} is proved by investigating the system \eqref{rewrite system}. We show that under a sufficiently small condition of the initial data, the associated solution of \eqref{rewrite system} scatters in both time directions and this implies the desired result.\\

This paper is organized as follows. In Section \ref{sec:Pre and notation}, we give some useful notation and preliminaries for the proof of our main results. Section \ref{sec:proof of the main results} is devoted to prove the main results of this paper Theorem \ref{thm1} and Theorem \ref{thm2}.

\section*{Acknowledgement} The author was supported by Post-doc fellowship of Labex MME-DII: SAIC/2022 No 10078. The author would like to thank an unknown reviewer for some useful discussions to improve this paper.

\section{Preliminaries and notation}
\label{sec:Pre and notation}

For convenience, we define $L=i\partial_t+\partial_x^2$ the Schr\"odinger operator.

For a function $f$, we denote by $f_x$ or $\partial_x f$ the derivative in space of the function $f$.

We denote by $e^{it\partial_x^2}$ the propagator of the linear Schr\"odinger equation.

Denote $|(a,b)|=|a|+|b|$ and $\norm{(a,b)}_{X}=\norm{a}_{X}+\norm{b}_X$, for any Banach space $X$.

A pair $(q,r)$ is called an admissible in dimension one if $q,r>0$ and $\frac{2}{q}+\frac{1}{r}=\frac{1}{2}$. We define $\mathcal{A}$ to be the set of all admissible pairs.

For a function $f(z)$ defined for a complex variable $z$ and for a positive integer $k$, we define $k$th order complex derivative of $f(z)$ by:
\[
f^{k}(z):=\left(\frac{\partial^k f}{\partial_z^k}, \frac{\partial^k}{\partial_z^{k-1}\partial_{\overline{z}}}, \cdots, \frac{\partial^k f}{\partial_z\partial_{\overline{z}}^{k-1}},\frac{\partial^k f}{\partial_{\overline{z}}^k}\right),
\]
where 
\[
\frac{\partial f}{\partial_z}:=\frac{1}{2}\left(\frac{\partial f}{\partial x}-i\frac{\partial f}{\partial y}\right), \quad \frac{\partial f}{\partial_{\overline{z}}}:=\frac{1}{2}\left(\frac{\partial f}{\partial x}+i\frac{\partial f}{\partial y}\right).
\]
We refer the reader to \cite{AnKi21} for related notations.

Let $I$ be an interval of $\R$. The Strichartz space $S(I)$ is defined by
\[
\norm{u}_{S(I)}=\sup_{(q,r)\in\mathcal{A}}\norm{u}_{L^qL^r}.
\]
Let $N(I)$ be the dual space of $S(I)$. Define $S^1,N^1$ by
\begin{align*}
\norm{u}_{S^1(I)} &=\norm{(u,\partial_x u)}_{S(I)}=\norm{\left<\nabla\right>u}_{S(I)},\\
\norm{u}_{N^1(I)} &=\norm{(u,\partial_x u)}_{N(I)}=\norm{\left<\nabla\right>u}_{N(I)}.
\end{align*}

Let $\varphi=G_1(u)$ and $\psi=G_2(u)$. Then, $\psi=\partial_x\varphi-\frac{i}{2}|\varphi|^{2\sigma}\varphi$.  As in \cite{Tinpaper4}, we see that if $u$ shows \eqref{eq1} then $(\varphi,\psi)$ show the following system:
\begin{equation}\label{eqsystem}
\begin{cases}
L\varphi&=P(\varphi,\psi):=i\sigma |\varphi|^{2(\sigma-1)}\varphi^2\overline{\psi}-\sigma(\sigma-1)\varphi\int_{-\infty}^x|\varphi|^{2(\sigma-2)}\Im(\psi^2\overline{\varphi}^2),\\
L\psi&=Q(\varphi,\psi):=-i\sigma |\varphi|^{2(\sigma-1)}\psi^2\overline{\varphi}-\sigma(\sigma-1)\psi\int_{-\infty}^x|\varphi|^{2(\sigma-2)}\Im(\psi^2\overline{\varphi}^2).
\end{cases}
\end{equation}
Let $\eta=(\varphi,\psi)$ and $F=(P,Q)$. We may rewrite \eqref{eqsystem} as follows:
\begin{equation}
\label{rewrite system}
L\eta=F(\eta).
\end{equation}

We next introduce useful tools for the proof of the main results in the next section.

\begin{lemma}\label{lm1}
For $\eta_1=(\varphi_1,\psi_1),\eta_2=(\varphi_2,\psi_2) \in H^1\times H^1$, we have
\begin{align*}
|F(\eta_1)-F(\eta_2)|&\lesssim |\eta_1-\eta_2|(|\eta_1|^{2\sigma}+|\eta_2|^{2\sigma}+\int_{-\infty}^x|\eta_1|^{2\sigma}\,dy)\\
&+|\eta_2|\int_{-\infty}^x|\eta_1-\eta_2|(|\eta_1|^{2\sigma-1}+|\eta_2|^{2\sigma-1})\,dy,\\
|\partial_x(F(\eta_1)-F(\eta_2))|&\lesssim |\partial_x(\eta_1-\eta_2)|\left(|\eta_1|^{2\sigma}+|\eta_2|^{2\sigma}+\int_{-\infty}^x|\eta_1|^{2\sigma}\right)\\
&\quad +|\eta_1-\eta_2||\partial_x(\eta_1,\eta_2)|(|\eta_1|^{2\sigma-1}+|\eta_2|^{2\sigma-1})\\
&+|\partial_x\eta_2|\int_{-\infty}^x|\eta_1-\eta_2|(|\eta_1|^{2\sigma-1}+|\eta_2|^{2\sigma-1})+|\eta_1-\eta_2|(|\eta_1|^{2\sigma}+|\eta_2|^{2\sigma}).
\end{align*}
\end{lemma}
To prove Lemma \ref{lm1}, we only need to prove the following estimates. 
\begin{lemma}
Let $\sigma>2$. We have the following estimates:
\begin{align*}
\left||\varphi_1|^{2(\sigma-1)}\varphi_1^2\overline{\psi}_1-|\varphi_2|^{2(\sigma-1)}\varphi_2^2\overline{\psi}_2\right|&\lesssim |\eta_1-\eta_2||(\eta_1,\eta_2)|^{2\sigma},\\
\left||\varphi_1|^{2(\sigma-2)}\Im(\psi_1^2\overline{\varphi}_1^2)-|\varphi_2|^{2(\sigma-2)}\Im(\psi_2^2\overline{\varphi}_2^2)\right|&\lesssim |\eta_1-\eta_2||(\eta_1,\eta_2)|^{2\sigma-1},\\
\left|\partial_x(|\varphi_1|^{2(\sigma-1)}\varphi_1^2\overline{\psi}_1-|\varphi_2|^{2(\sigma-1)}\varphi_2^2\overline{\psi}_2)\right|&\lesssim |\partial_x(\eta_1-\eta_2)||(\eta_1,\eta_2)|^{2\sigma},
\end{align*}
where $\eta_1=(\varphi_1,\psi_1),\eta_2=(\varphi_2,\psi_2)$.
\end{lemma}

\begin{proof}
For each function $f\in C^1$, we have the following expression (see for instance \cite{Vi07}):
\begin{equation}
\label{eq:expression}
f(u)-f(v)=(u-v)\int_0^1 \partial_z f(v+\theta (u-v))d\theta + \overline{u-v}\int_0^1 \partial_{\overline{z}}f(v+\theta (u-v))d\theta.
\end{equation}
Using the above expression for $f(z)=|z|^{2(\sigma-1)}$ with noting that $f\in C^1$ (since $\sigma>2$) and $|f^{(1)}(z)|\lesssim |z|^{2(\sigma-1)-1}$, we have
\[
\left||\varphi_1|^{2(\sigma-1)}-|\varphi_2|^{2(\sigma-1)}\right|=|f(\varphi_1)-f(\varphi_2)|\lesssim |\varphi_1-\varphi_2||(\varphi_1,\varphi_2)|^{2(\sigma-1)-1}.
\] 
This implies easily the first estimate. To prove the second estimate, we using the expression \eqref{eq:expression} for $f(z)=|z|^{2(\sigma-2)}\overline{z}^2$. As $f\in C^1$ (since $\sigma>2$) and $|f^{(1)}(z)|\lesssim |z|^{2(\sigma-1)-1}$, we have 
\[
\left||\varphi_1|^{2(\sigma-2)}\overline{\varphi_1}^2-|\varphi_2|^{2(\sigma-2)}\overline{\varphi_2}^2\right|\lesssim |\varphi_1-\varphi_2||(\varphi_1,\varphi_2)|^{2(\sigma-1)-1}.
\]
This implies easily the second estimate. Consider the third estimate. The most difficult contribution in the third estimate is the following 
\[
\left|\partial_x(|\varphi_1|^{2(\sigma-1)}\varphi_1^2-|\varphi_2|^{2(\sigma-1)}\varphi_2^2)\right|.
\]
From \eqref{eq:expression}, we have
\begin{equation}
\label{eq:derivative expression}
|\partial_x(f(u)-f(v))|\lesssim |\partial_x(u-v)|\sup_{\theta\in [0,1]}|f^{(1)}(v+\theta (u-v))|+ |u-v| \sup_{\theta\in [0,1]} |\partial_x f^{(1)}(v+\theta (u-v))|.
\end{equation}
We use the above expression for $f(z)=|z|^{2(\sigma-1)}z^2$. The first term is easy to handle. Consider the second term. We have
\[
\partial_z f (z)= (\sigma+1)|z|^{2(\sigma-1)}z,\quad \partial_{\overline{z}}f(z)=(\sigma-1)|z|^{2(\sigma-2)}z^3.
\]
Thus, 
\[
|\partial_x \partial_zf(u)|\lesssim |\partial_xu||u|^{2(\sigma-1)}, \quad |\partial_x \partial_{\overline{z}}f(u)|\lesssim |\partial_x u||u|^{2(\sigma-1)}.
\]
From this, we easily estimate the second term in \eqref{eq:derivative expression} and the third estimate is proved.
\end{proof}

\section{Proof of the main results}
\label{sec:proof of the main results}

\subsection{Proof of Theorem \ref{thm1}}
\label{sec:Proof 1}

The aim of this section is to prove Theorem \ref{thm1}. We use the similar argument as in \cite{Tinpaper4}.\\
Let $R=e^{it\partial_x^2}u^+$. Then
\[
LR+i|R|^{2\sigma}R_x=i|R|^{2\sigma}R_x:=v(t,x).
\]
The Gauge transformations of $R$ are defined by:
\begin{align*}
h&=G_1(R),\\
k&=G_2(R)=\partial_x h-\frac{i}{2}|h|^{2\sigma}h.
\end{align*}
We prove that
\begin{align}
Lh&=P(h,k)+m\label{eqof h},\\
Lk&=Q(h,k)+n\label{eqof k},
\end{align}
where $m,n$ are functions depending on $R,v$:
\begin{align*}
m(t,x)&=\exp\left(\frac{i}{2}\int_{-\infty}^x |R(t,y)|^2dy\right)v(t,x),\\
n&=\partial_xm-\frac{i}{2}(\sigma+1)|h|^{2\sigma}m-\frac{i}{2}\sigma |h|^{2(\sigma-1)}h^2\overline{m}.
\end{align*}
Indeed, since $h(t,x)=\exp\left(\frac{i}{2}\int_{-\infty}^x |R(t,y)|^{2\sigma}dy\right)R(t,x)$, we have
\begin{align*}
\partial_t h&=\exp\left(\frac{i}{2}\int_{-\infty}^x |R(t,y)|^{2\sigma}dy\right)\left(\partial_t R+\frac{i}{2}R\partial_t\int_{-\infty}^x|R(t,y)|^{2\sigma}dy\right),\\
\partial_{xx} h&=\exp\left(\frac{i}{2}\int_{-\infty}^x |R(t,y)|^{2\sigma}dy\right)\left(R_{xx}+i|R|^{2\sigma}R_x+\frac{i}{2}\partial_x(|R|^{2\sigma})R-\frac{1}{4}|R|^{4\sigma}R\right).
\end{align*}   
Moreover, as in \cite{HaOz16}[Section 4], we have
\begin{align*}
\frac{1}{2}\partial_t \int_{-\infty}^x |R|^{2\sigma}dy&=-\sigma |h|^{2(\sigma-1)}\Im(\overline{h}k)+\sigma \int_{-\infty}^x\partial_x(|h|^{2(\sigma-1)})\Im(\overline{h}k)dy-\frac{1}{4}|h|^{4\sigma}.
\end{align*}
Thus,
\begin{align*}
Lh&=\exp\left(\frac{i}{2}\int_{-\infty}^x |R(t,y)|^{2\sigma}dy\right)\left(LR+i|R|^{2\sigma}R_x-\frac{1}{2}R\partial_t\int_{-\infty}^x|R|^{2\sigma}dy+\frac{i}{2}R\partial_x(|R|^{2\sigma})-\frac{1}{4}|R|^{4\sigma}R\right)\\
&=\exp\left(\frac{i}{2}\int_{-\infty}^x |R(t,y)|^{2\sigma}dy\right) v+ h\left(-\frac{1}{2}\partial_t\int_{-\infty}^x|R|^{2\sigma}dy +\frac{i}{2}\partial_x(|h|^{2\sigma})-\frac{1}{4}|h|^{4\sigma}\right)\\
&=\exp\left(\frac{i}{2}\int_{-\infty}^x |R(t,y)|^{2\sigma}dy\right) v+h\left(\sigma |h|^{2(\sigma-1)}\Im(\overline{h}k)-\sigma \int_{-\infty}^x \partial_x(|h|^{2(\sigma-1)})\Im(\overline{h}k)dy+\frac{i}{2}\partial_x(|h|^{2\sigma})\right)\\
&=\exp\left(\frac{i}{2}\int_{-\infty}^x |R(t,y)|^{2\sigma}dy\right) v+h\left(i\sigma |h|^{2(\sigma-1)}h\overline{k}-\sigma(\sigma-1)\int_{-\infty}^x|h|^{2(\sigma-2)}\Im(\overline{h}^2 k^2)dy\right)\\
&=\exp\left(\frac{i}{2}\int_{-\infty}^x |R(t,y)|^{2\sigma}dy\right) v+P(h,k)\\
&=P(h,k)+m,
\end{align*}
which gives \eqref{eqof h}. Furthermore, using $k=\partial_xh-\frac{i}{2}|h|^{2\sigma}h$, we have 
\begin{align*}
Lk&=\partial_x Lh-\frac{i}{2}L(|h|^{2\sigma}h)\\
&=\partial_x Lh- \frac{i}{2}\left(i\partial_t (h^{\sigma+1}\overline{h}^{\sigma}) +\partial_{xx}(h^{\sigma+1}\overline{h}^{\sigma})\right)\\
&=\partial_x Lh\\
&\quad -\frac{i}{2}\left((\sigma+1)|h|^{2\sigma}i\partial_t h+\sigma |h|^{2(\sigma-1)} h^2 i\partial_t\overline{h} + (\sigma+1)|h|^{2\sigma}\partial_{xx}h +\sigma |h|^{2(\sigma-1)}h^2\partial_{xx}\overline{h}\right.\\
&\quad \quad \left. +(\sigma+1)\partial_xh \partial_x(|h|^{2\sigma})+\sigma \partial_x\overline{h}\partial_x(|h|^{2(\sigma-1)}h^2) \right)\\ 
&=\partial_xLh -\frac{i}{2}\left((\sigma+1)|h|^{2\sigma}Lh + \sigma |h|^{2(\sigma-1)}h^2 (2\partial_{xx} \overline{h}-\overline{Lh})\right)\\
&\quad -\frac{i}{2}(\sigma+1)\partial_xh\partial_x(|h|^{2\sigma})-\frac{i}{2}\sigma \partial_x\overline{h}\partial_x(|h|^{2(\sigma-1)}h^2).
\end{align*}
In the above expression, the contribution of $\partial_{xx}\overline{h}$ can be removed from the contribution of $\partial_x P(h,k)$ in $\partial_xLh$. The contribution of $m$ in the above expression of $Lk$ appears only in the terms $\partial_x Lh$, $Lh$ and $\overline{Lh}$. It is easy to check that this equals to 
$$\partial_x m -\frac{i}{2}(\sigma+1)|h|^{2\sigma}m-\frac{i}{2}\sigma |h|^{2(\sigma-1)}h^2\overline{m}.$$
Since the remainder depends only on $h,k$ and does not depend on $m$, which is actually a polynomial of $h,k,\overline{h},\overline{k}$, by using the similar argument to obtain the system \ref{eqsystem}, we have
\[
Lk=\partial_x m -\frac{i}{2}(\sigma+1)|h|^{2\sigma}m-\frac{i}{2}\sigma |h|^{2(\sigma-1)}h^2\overline{m}+ Q(h,k)=Q(h,k)+n,
\]
which gives \eqref{eqof k}. By H\"older's inequality and dispersive estimates, we have, for $t\geq 1$, 
$$\norm{(m,n)}_{H^1}\lesssim \norm{v}_{H^2}\lesssim t^{-\sigma}\norm{u^+}_{W^{1,1} \cap H^3}^{2\sigma+1}.$$

Let $\hat{W}=(h,k)$, $\hat{H}=-(m,n)$ and $\tilde{\eta}=\eta-\hat{W}=(\tilde{\varphi},\tilde{\psi})$. Then, 
\begin{equation}
\label{relation tilde eta}
\partial_x\tilde{\psi}=\partial_x\tilde{\varphi}-\frac{i}{2}(|\tilde{\varphi}+h|^{2\sigma}(\tilde{\varphi}+h)-|h|^{2\sigma}h).
\end{equation}
We see that if $u$ shows \eqref{eq1} then $\tilde{\eta}$ shows the following system
\begin{equation}
\label{system new}
L\tilde{\eta}=F(\tilde{\eta}+\hat{W})-F(\hat{W})+\hat{H}.
\end{equation}
We divide the proof of Theorem \ref{thm1} into three steps. \\
\emph{Step 1. Existence of a solution of the system}

In this step, we show that \eqref{system new} admits a solution which decays polynomially in time. To do so, we need the following result.
\begin{proposition}\label{pro exist solution}
Let $F=(P,Q)$, where $P,Q$ are defined by \eqref{eqsystem}. Let $H=H(t,x):[0,\infty)\times\R\rightarrow\C^2$, $W=W(t,x):[0,\infty)\times\R\rightarrow\C^2$ be given functions which satisfy for some $C_1,C_2>0$, $T_0>0$:
\begin{align}
t^{\frac{1}{2}}\norm{W(t)}_{L^{\infty}\times L^{\infty}}+\norm{W(t)}_{L^2 \times L^2}+t^{\sigma}\norm{H(t)}_{L^2 \times L^2}&\leq C_1,\quad \forall t\geq T_0,\label{eq system 1}\\
t^{\frac{1}{2}}\norm{\partial_xW(t)}_{L^{\infty}\times L^{\infty}}+\norm{\partial_x W(t)}_{L^{2}\times L^2}+t^{\sigma}\norm{\partial_xH(t)}_{L^2 \times L^2}&\leq C_2,\quad \forall t\geq T_0.\label{eq system 2}
\end{align}
Then the system
\begin{equation}
\label{eq:system prop 3.1}
L\tilde{\eta}=F(\tilde{\eta}+W)-F(W)+H
\end{equation}
admits a unique solution $\tilde{\eta}$ that satisfies
\begin{equation}
\label{eq:estimate of tilde_eta}
t^{\sigma-1}(\norm{\tilde{\eta}}_{S[t,\infty)\times S[t,\infty)}+\norm{\partial_x\tilde{\eta}}_{S[t,\infty)\times S[t,\infty)})\lesssim 1,
\end{equation}
for all $t\geq T_0$.
\end{proposition}

\begin{proof}
Fix $T_0>1$ large enough. Define
\[
B=\left\{\tilde{\eta}\in C([T_0,\infty),H^1\times H^1):t^{\sigma-1}\norm{(\tilde{\eta},\partial_x\tilde{\eta})}_{S[t,\infty)\times S[t,\infty)}\leq C, \quad\forall t\geq T_0\right\},
\]
for some $C>1$ large enough and 
\[
\Phi(\tilde{\eta})(t)=i\int_t^{\infty} e^{i(t-s)\partial_x^2}(F(\tilde{\eta}+W)-F(W)+H)(s)\,ds.
\]
If $\tilde{\eta}$ tends to zero in $H^1$-norm as $t$ tends to infinity then the system \eqref{eq:system prop 3.1} is rewritten as follows
$$\tilde{\eta}=\Phi(\tilde{\eta}).$$
Thus, it suffices to prove that $\Phi$ is a contraction mapping on $B$ for $T_0$ large enough. We divide the proof in two steps.\\
\emph{Step 1. Proof $\Phi$ maps $B$ into $B$}\\
Let $t\geq T_0$. For convenience, let $\tilde{\eta}=(\tilde{\eta_1},\tilde{\eta}_2) \in B$, $W=(w_1,w_2)$ and $H=(h_1,h_2)$. By Strichartz's estimates, we have
\begin{align}
\norm{\Phi\tilde{\eta}}_{S[t,\infty)\times S[t,\infty)}&\lesssim \norm{F(\tilde{\eta}+W)-F(W)}_{N[t,\infty)\times N[t,\infty)}\label{eq10}\\
&\quad+\norm{H}_{L^1L^2([t,\infty))\times L^1L^2([t,\infty))}\label{eq100}.
\end{align} 
To estimate \eqref{eq100}, using \eqref{eq system 1}, we have
\[
\norm{H}_{L^1L^2([t,\infty))\times L^1L^2([t,\infty))}\lesssim \int_t^{\infty}s^{-\sigma}\,ds=t^{1-\sigma}.
\]
We have 
\[
\norm{\tilde{\eta}(s)}_{L^{2\sigma}}\lesssim \norm{\tilde{\eta}(s)}_{H^1}\lesssim s^{1-\sigma},
\]
for all $s\in [t,\infty)$. It implies that 
\[
\norm{\tilde{\eta}}^{2\sigma}_{L^{2\sigma}_{t,x}(t,\infty)}\lesssim \int_t^{\infty} s^{2\sigma(1-\sigma)}ds \lesssim t^{1+2\sigma (1-\sigma)}\lesssim t^{1-2\sigma}.
\]
Moreover,
\[
\norm{\tilde{\eta}}_{L^1L^2(t,\infty)}\lesssim \int_t^{\infty} s^{1-\sigma}ds=t^{2-\sigma}.
\]
Hence, using Lemma \ref{lm1}, Hölder's inequality and the estimate $\left|\int_{-\infty}^xf(y)dy\right|\leq\norm{f}_{L^1}$, we have 
\begin{align*}
&\norm{F(\tilde{\eta}+W)-F(W)}_{N[t,\infty)\times N[t,\infty)}\\
&\lesssim \norm{F(\tilde{\eta}+W)-F(W)}_{L^1L^2[t,\infty)\times L^1L^2[t,\infty)} \\
&\lesssim \norm{\tilde{\eta}}_{L^1L^2}\left(\norm{(\tilde{\eta},W)}_{L^{\infty}L^{\infty}}^{2\sigma}+\norm{(\tilde{\eta},W)}_{L^{\infty}L^{2\sigma}}^{2\sigma}\right)\\
&\quad +\norm{W}_{L^{\infty}L^2}\norm{|\tilde{\eta}||(\tilde{\eta},W)|^{2\sigma-1}}_{L^1L^1}\\
&\lesssim \norm{\tilde{\eta}}_{L^1L^2}\left(\norm{(\tilde{\eta},W)}_{L^{\infty}L^{\infty}}^{2\sigma}+\norm{\tilde{\eta}}_{L^{\infty}L^{2\sigma}}^{2\sigma}+\norm{W}_{L^{\infty}L^2}^2\norm{W}^{2\sigma-2}_{L^{\infty}L^{\infty}}\right)\\
&\quad +\norm{W}_{L^{\infty}L^2}\left(\norm{\tilde{\eta}}_{L^{2\sigma}_{t,x}}^{2\sigma}+\norm{\tilde{\eta}}_{L^1L^2}\norm{W}_{L^{\infty}L^2}\norm{W}^{2\sigma-2}_{L^{\infty}L^{\infty}}\right)\\
&\lesssim t^{1-\sigma}.
\end{align*}
This implies that
\[
\norm{\Phi\tilde{\eta}}_{S[t,\infty)\times S[t,\infty)}\leq C t^{1-\sigma},
\]
for some constant $C>0$. Moreover,
\begin{align}
\norm{\partial_x\Phi\tilde{\eta}}_{S[t,\infty)\times S[t,\infty)}&\lesssim \norm{\partial_x(F(\tilde{\eta}+W)-F(W))}_{N[t,\infty)\times N[t,\infty)}\label{eq 221}\\
&\quad +\norm{\partial_x H}_{L^1L^2[t\,\infty)\times L^1L^2[t,\infty)}\label{eq222}.
\end{align}
To estimate \eqref{eq222}, using \eqref{eq system 2}, we have
\[
\norm{\partial_x H}_{L^1L^2([t,\infty))\times L^1L^2([t,\infty))}
\lesssim \int_t^{\infty}s^{-\sigma}\,ds=t^{1-\sigma}.
\]
Furthermore, using Lemma \ref{lm1}, we have
\begin{align*}
&\norm{\partial_x(F(\tilde{\eta}+W)-F(W))}_{N[t,\infty)\times N[t,\infty)}\\
&\lesssim\norm{\partial_x\tilde{\eta}}_{L^1L^2}(\norm{(\tilde{\eta},W)}^{2\sigma}_{L^{\infty}L^{\infty}}+\norm{\tilde{\eta}}^{2\sigma}_{L^{\infty}L^{2\sigma}}+\norm{W}_{L^{\infty}L^2}^2\norm{W}_{L^{\infty}L^{\infty}}^{2\sigma-2})\\
&\quad +\left(\norm{\tilde{\eta}}_{L^1L^2}\norm{\partial_x W}_{L^{\infty}L^{\infty}}+\norm{\tilde{\eta}}_{L^{\frac{4}{3}}L^2}\norm{\partial_x\tilde{\eta}}_{L^4L^{\infty}}\right)\norm{(\tilde{\eta},W)}^{2\sigma-1}_{L^{\infty}L^{\infty}}\\
&\quad +\norm{\partial_xW}_{L^{\infty}L^2}\norm{\tilde{\eta}}_{L^1L^2}\norm{(\tilde{\eta},W)}_{L^{\infty}L^2}\norm{(\tilde{\eta},W)}_{L^{\infty}L^{\infty}}^{2\sigma-2}+\norm{\tilde{\eta}}_{L^1L^2}\norm{(\tilde{\eta},W)}^{2\sigma}_{L^{\infty}L^{\infty}}\\
&\lesssim t^{1-\sigma}.
\end{align*}
Thus, $\Phi$ maps $B$ into $B$ if we chose $C$ large enough. \\
\emph{Step 2. Proof $\Phi$ is a contraction map on $B$}

By using \eqref{eq system 1}, \eqref{eq system 2} and a similar argument as in Step 1, we show that, for any $\tilde{\eta},\tilde{\nu}\in B$, we have
\[
\norm{\Phi\tilde{\eta}-\Phi\tilde{\nu}}_X \leq \frac{1}{2}\norm{\tilde{\eta}-\tilde{\nu}}_X,
\]
where
\[
\norm{\tilde{\eta}}_X=\sup_{t\geq T_0} t^{\sigma-1}\norm{(\tilde{\eta},\partial_x\tilde{\eta})}_{S[t,\infty)\times S[t,\infty)}.
\]
Thus, $\Phi$ is a contraction map on $B$. This implies the desired result.
\end{proof}

By the dispersive estimate, we have $\norm{R}_{L^{\infty}}\lesssim t^{\frac{-1}{2}}\norm{u^+}_{L^1}$. Thus, it is easy to check that the functions $\hat{W},\hat{H}$ satisfy the conditions \eqref{eq system 1} and \eqref{eq system 2}. Hence, using Proposition \ref{pro exist solution}, there exists a unique solution $\tilde{\eta}$ to the system \eqref{system new} satisfying the estimate \eqref{eq:estimate of tilde_eta}.

\emph{Step 2. Existence of a desired solution}\\

In this step, we show that the solution $\tilde{\eta}$ given in Step 1 satisfies a desired relation. This implies the existence of the desired solution of \eqref{eq1}. More precisely, we have the following result:
\begin{lemma}
\label{Lm:relation of solution to the system}
Let $\tilde{\eta}=(\tilde{\varphi},\tilde{\psi})$ be the solution of \eqref{system new} given in Step 1. Then,
\[
\tilde{\psi}=\partial_x\tilde{\varphi}-\frac{i}{2}(|\tilde{\varphi}+h|^{2\sigma}(\tilde{\varphi}+h)-|h|^{2\sigma}h).
\] 
\end{lemma} 

\begin{proof}
Define $(\varphi,\psi)=(\tilde{\varphi},\tilde{\psi})+(h,k)$. Let $\tilde{\kappa}=\partial_x\tilde{\varphi}-\frac{i}{2}(|\tilde{\varphi}+h|^{2\sigma}(\tilde{\varphi}+h)-|h|^{2\sigma}h)$ and $\kappa=\tilde{\kappa}+k=\partial_x\varphi-\frac{i}{2}|\varphi|^2\varphi$. Hence, we only need to show that $\tilde{\kappa}=\tilde{\psi}$. To do so, we use the idea used in \cite{Tinpaper4} (see also \cite{HaOz94,HaOz92,Ha93}, where the authors constructed solutions by given initial data). We will consider the quantity $\norm{\tilde{\psi}(t)-\tilde{\kappa}(t)}_{L^2}^2$ and estimate the invariant of this in time. It seems that this quantity is better to use than $\norm{\psi(t)-\kappa(t)}_{L^2}^2$ since we already have the boundedness of $\tilde{\psi},\tilde{v}$ in large time by Proposition \ref{pro exist solution}. As in \cite{Tinpaper4}, we have for $N\gg t$:
\begin{align*}
L\tilde{\psi}-L\tilde{\kappa}&= L\psi-L\kappa\\
&=Q(\varphi,\tilde{\psi}+k)-Q(\varphi,\tilde{\kappa}+k)-\partial_x(P(\varphi,\tilde{\psi}+k)-P(\varphi,\tilde{\kappa}+k)\\
&\quad+\frac{i}{2}(\sigma+1)|\varphi|^{2\sigma}(P(\varphi,\tilde{\psi}+k)-P(\varphi,\tilde{\kappa}+k))\\
&\quad-\frac{i}{2}\sigma|\varphi|^{2(\sigma-1)}\varphi^2(\overline{P(\varphi,\tilde{\psi}+k)}-\overline{P(\varphi,\tilde{\kappa}+k)}).
\end{align*}
Multiplying both side of the above equality by $\overline{\tilde\psi-\tilde{v}}$, taking imaginary part and integrating over space with integration by parts (see for instance \cite{Tinpaper4}) we obtain
\begin{align}
\norm{\tilde{\psi}(t)-\tilde{\kappa}(t)}_{L^2}^2&\lesssim \norm{\tilde{\psi}(N)-\tilde{\kappa}(N)}_{L^2}^2\exp\left(\int_t^N (K_1+K_2+K_3)(s)\,ds\right),\label{eq 99}
\end{align}
where 
\begin{align*}
K_1&=\norm{\varphi}^{2\sigma-1}_{L^{\infty}}\norm{\tilde{\psi}+\tilde{\kappa}+2k}_{L^{\infty}}+\norm{|\varphi|^{2(\sigma-1)}(\tilde{\psi}+k)^2}_{L^1} +\norm{\tilde{\kappa}+k}_{L^2}\norm{|\varphi|^{2(\sigma-1)}(\tilde{\psi}+\tilde{\kappa}+2k)}_{L^2},\\
K_2&= \norm{\partial_x(|\varphi|^{2(\sigma-1)}\varphi^2)}_{L^{\infty}}+\norm{\partial_x\varphi}_{L^2}\norm{|\varphi|^{2(\sigma-1)}(\tilde{\psi}+\tilde{\kappa}+2k)}_{L^2}+\norm{|\varphi|^{2\sigma-1}(\tilde{\psi}+\tilde{\kappa}+2k)}_{L^{\infty}},\\
K_3&= \norm{\varphi}^{4\sigma}_{L^{\infty}}+\norm{|\varphi|^{2\sigma+1}}_{L^2}\norm{|\varphi|^{2(\sigma-1)}(\tilde{\psi}+\tilde{\kappa}+2k)}_{L^2}.
\end{align*}
Consider $K_1$. To estimate this term, we use the boundedness of $\tilde{\eta}$ as in Proposition \ref{pro exist solution} and the well-known boundedness  of $(h,k)$. Since, $\norm{\varphi(s)}_{L^{\infty}}\lesssim \norm{\tilde{\varphi}(s)}_{L^{\infty}}+\norm{h(s)}_{L^{\infty}}\lesssim s^{\frac{-1}{2}}$, we have
\begin{align*}
K_1&\lesssim \norm{\varphi}^{2\sigma-1}_{L^{\infty}}(\norm{\tilde{\kappa}}_{L^{\infty}}+\norm{\tilde{\psi}+2k}_{L^{\infty}})+\norm{\varphi}_{L^{\infty}}^{2(\sigma-1)}\norm{\tilde{\psi}+k}^2_{L^2}+\norm{\tilde{\kappa}+k}_{L^2}\norm{\varphi}_{L^{\infty}}^{2(\sigma-1)}\norm{\tilde{\psi}+\tilde{\kappa}+2k}_{L^2}\\
&\lesssim s^{\frac{1}{2}-\sigma}(\norm{\partial_x\tilde{\varphi}}_{L^{\infty}}+s^{-\frac{1}{2}})+s^{1-\sigma}.
\end{align*}
To estimate the contribution of the term $\norm{\partial_x\tilde{\varphi}}_{L^{\infty}}$ in $K_1$, we shall use Strichartz norm of this term. Since $\sigma>2$, we have, for all $N>t$,
\begin{align*}
\int_t^N K_1(s)ds&\lesssim \norm{K_1}_{L^1(t,\infty)}\\
&\lesssim \norm{s^{\frac{1}{2}-\sigma} \norm{\partial_x\tilde{\varphi}(s)}_{L^{\infty}}}_{L^1(t,\infty)}+\norm{s^{1-\sigma}}_{L^1(t,\infty)}\\
&\lesssim \norm{s^{\frac{1}{2}-\sigma}}_{L^{\frac{4}{3}}(t,\infty)}\norm{\partial_x\tilde{\varphi}}_{L^4(t,\infty)L^{\infty}_x}+t^{2-\sigma}\\
&\lesssim t^{\frac{5-4\sigma}{3}}t^{1-\sigma}+t^{2-\sigma}<\infty.
\end{align*}
Similarly, we show that $K_2,K_3 \in L^1(t,\infty)$, where the contribution of $\norm{\partial_x\tilde{\varphi}}_{L^{\infty}}$ in $K_2$ are estimated by using the Strichartz norm. Moreover, using \eqref{eq:estimate of tilde_eta}, $\norm{\tilde{\psi}(N)-\tilde{\kappa}(N)}_{L^2}\lesssim N^{1-\sigma}\rightarrow 0$ as $N\rightarrow\infty$. Thus, fixing $t$, the right hand side of \eqref{eq 99} converges to zero as $N$ tends to infinity. It implies that $\tilde{\psi}=\tilde{\kappa}$. This completes the proof of Lemma \ref{Lm:relation of solution to the system}. 
\end{proof}

Let $u=G_1^{-1}(\tilde{\varphi}+h)$. From Lemma \ref{Lm:relation of solution to the system}, we obtain $u$ is a solution of \eqref{eq1}. Moreover, we have
\[
\norm{u-R}_{H^1}\lesssim \norm{\tilde{\varphi}}_{H^1}\lesssim t^{1-\sigma}\rightarrow 0.
\]
This implies that $u$ scatters forward to $u^+$. It is easy to see that $u \in C([T_0,\infty),H^1) \cap L^4([T_0,\infty),W^{1,\infty})$. Moreover, $\norm{\partial_x u}_{L^{\infty}_t([T_0,\infty),L^{\infty}_x(\R))}=\norm{\psi}_{L^{\infty}_t([T_0,\infty),L^{\infty}_x(\R))}$ is well bounded.\\

Now, assume that $u^+$ satisfies the condition \eqref{eq:condition on u^+}. Since $u^+\in H^3\cap W^{2,1}$, we have $\norm{R}_{L^{\infty}} \rightarrow 0$ as $t\rightarrow\infty$ and hence 
$$\Re\int_{\R}i|R|^{2\sigma}\partial_xR \overline{R}dx\rightarrow 0,\quad \text{as } t\rightarrow\infty.$$
Thus,
\begin{align*}
\lim_{t\rightarrow\infty}S_{\omega,0}(R(t))&= \frac{1}{2}\norm{\partial_xu^+}^2_{L^2}+\frac{\omega}{2}\norm{u^+}^2_{L^2}<\mu(\omega,0),\\
\lim_{t\rightarrow\infty}K_{\omega,0}(R(t))&=\norm{\partial_xu^+}^2_{L^2}+\omega\norm{u^+}^2_{L^2} >\omega\norm{u^+}^2_{L^2}>0.
\end{align*}
Since $\norm{u(t)-R(t)}_{H^1}\rightarrow 0$ as $t\rightarrow\infty$, we have $K_{\omega,0}(u(t))\geq \omega\norm{u^+}^2_{L^2}> 0$ and $S_{\omega,0}(u(t))<\mu(\omega,0)$ for $t$ large enough. By Theorem \ref{thm:global}, $u$ exists globally in time. Moreover, by Theorem \ref{thm:LWP in H_1}, $u\in C(\R,H^1)\cap L^4_{\text{loc}}(\R,W^{1,\infty})$. It remains to show that the existence of $u$ is unique in some sense. 

\emph{Step 3. Uniqueness of the desired solution}\\

In this step, we show that the solution found in Step 2 is unique in some sense, which implies that the wave map is well defined. We need the following well known result.
\begin{theorem}[\cite{HaOz16}[Theorem 1.4]
\label{thm:LWP in H_1}
Let $\sigma\geq 1$ and $u_0\in H^1(\R)$. Then there exists $T>0$ and a unique solution $u\in C([-T,T],H^1)\cap L^4((-T,T),W^{1,\infty})$ of \eqref{eq1}. Moreover, $u$ satisfies the following properties:
\begin{itemize}
\item[(i)] $u\in L^q((-T,T),W^{1,r})$ for every admissible pair $(q,r)$.
\item[(ii)] $M(u(t))=M(u(0))$ and $E(u(t))=E(u(0))$ for all $t\in [-T,T]$.
\item[(iii)] $u$ depends continuously on $u_0$ in the following sense. If $u_0^n\rightarrow u_0$ in $H^1$ and if $u_n$ is the corresponding solution of \eqref{eq1} then $u_n$ is defined on $[-T,T]$ for $n$ large enough and  $u_n\rightarrow u$ in $C([-T,T],H^1)$.
\end{itemize}  
\end{theorem}
Using the above theorem, we have the following result.
\begin{lemma}
\label{lm:uniqueness of desired solution}
Let $u,w\in C(\R,H^1)\cap L^4_{\text{loc}}(\R,W^{1,\infty})$ be two solutions of \eqref{eq1} such that $\norm{(u_x,w_x)}_{L^{\infty}_t([T_0,\infty),L^{\infty}_x(\R))}\lesssim 1$ and
$$\sup_{t\geq T_0} t^{\sigma-1}(\norm{u(t)-R(t)}_{H^1}+\norm{w(t)-R(t)}_{H^1})\lesssim 1.$$
Then $u(t)=w(t)$ for all $t\in\R$.
\end{lemma}

\begin{proof}
We have 
$$i(u-w)_t+\partial_x^2 (u-w)+i(|u|^{2\sigma}u_x-|w|^{2\sigma}w_x)=0.$$
Multiplying two sides of the above equation by $\overline{(u-w)}$, taking imaginary part and integrating by space, we have
$$0=\frac{1}{2}\partial_t \norm{u-w}^2_{L^2}+\Re\int_{\R} (|u|^{2\sigma}u_x-|w|^{2\sigma}w_x)\overline{(u-w)}dx.$$
This implies that
\begin{align*}
\frac{1}{2}\partial_t \norm{u-w}^2_{L^2}&=-\Re\int_{\R} (|u|^{2\sigma}u_x-|w|^{2\sigma}w_x)\overline{(u-w)}dx\\
&=-\Re\int_{\R}(|u|^{2\sigma}-|w|^{2\sigma})u_x\overline{u-w}+|w|^{2\sigma}\overline{u-w}(u_x-w_x)dx\\
&=-\Re\int_{\R}(|u|^{2\sigma}-|w|^{2\sigma})u_x\overline{u-w}dx-\frac{1}{2}\Re\int_{\R}|w|^{2\sigma}\partial_x(|u-w|^2)dx\\
&=-\Re\int_{\R}(|u|^{2\sigma}-|w|^{2\sigma})u_x\overline{u-w}dx+\frac{1}{2}\Re\int_{\R}\partial_x(|w|^{2\sigma})|u-w|^2dx.
\end{align*}
Hence,
\begin{align*}
\left|\partial_t \norm{u-w}^2_{L^2}\right|&\lesssim \int_{\R} |u-w|^2(|u|^{2\sigma-1}+|w|^{2\sigma-1})(|u_x|+|w_x|)dx\\
&\lesssim \norm{u-w}^2_{L^2}\norm{(u,w)}^{2\sigma-1}_{L^{\infty}_x}\norm{(u_x,w_x)}_{L^{\infty}_x}\\
&\lesssim \norm{u-w}^2_{L^2} t^{\frac{1}{2}-\sigma}
\end{align*}
Define $M(t)=\norm{u(t)-w(t)}^2_{L^2}$. Assume $M(T_0)> 0$. We have 
\begin{align*}
\left|\int_{T_0}^b \frac{\partial_t M}{M}ds \right|&\lesssim \int_{T_0}^b \left|\frac{\partial_t M}{M}ds\right|\lesssim \int_{T_0}^b s^{\frac{1}{2}-\sigma}ds\\
\end{align*}
This implies that
$$|\log(M(b))-\log(M(T_0))|\lesssim  T_0^{\frac{3}{2}-\sigma} - b^{\frac{3}{2}-\sigma}.$$
Since $\sigma>2$ and $M(b)$ tends to zero as $b$ tends to infinity, the above estimate gives us an contradiction by letting $b$ tends to infinity. Thus, $M(T_0)=0$ and hence, $u(T_0)=w(T_0)$. Similarly, we have $u(t)=w(t)$ for all $t\geq T_0$. By the existence and uniqueness of $C(\R,H^1)\cap L^4_{\text{loc}}(\R,W^{1,\infty})$ solutions to \eqref{eq1} stated in Theorem \ref{thm:LWP in H_1}, we have $u(t)= w(t)$ for all $t\in\R$. This completes the proof.
\end{proof}
By Lemma \ref{lm:uniqueness of desired solution}, the wave map $u^+\mapsto u_0=u(0)$ is well defined. This completes the proof of Theorem \ref{thm1}.

\subsection{Proof of Theorem \ref{thm2}}
\label{sec:Proof 2}

Let $\sigma\geq 3$. To prove Theorem \ref{thm2}, it suffices to prove that the solution $\eta$ to \eqref{rewrite system} scatters in both time directions provided $\eta(0)$ is small enough in $H^1\times H^1$. Consider the system \eqref{system new}. Assume that $\eta(0)$ is sufficiently small in $H^1\times H^1$. By Strichartz's estimates, we need only to prove that $F(\eta) \in N^1(\R) \times N^1(\R)$. Fix $T>0$ and let $I=[0,T]$. From now on, we take all space-time estimates on $I$ unless otherwise stated. By Strichartz's estimates, we have
\[
\norm{\eta}_{S^1\times S^1}\lesssim \norm{\eta(0)}_{H^1\times H^1}+\norm{F(\eta)}_{N^1\times N^1}.
\]
Using Lemma \ref{lm1}, we have
\[
|F(\eta)|\lesssim |\eta|^{2\sigma+1}+|\eta|\left|\int_{\infty}^x |\eta|^{2\sigma}\right|.
\]
Thus, noting that $S^1(\R) \hookrightarrow L^{2\sigma}_{t,x}(\R\times\R)$ when $\sigma\geq 3$, $\norm{|\eta|}_{L^{\infty}_{t,x}}\lesssim \norm{|\eta|}_{S^1}$, $\norm{\int_{-\infty}^x f(y) dy}_{L^{\infty}_x}\leq \norm{f}_{L^1}$, we have
\begin{align*}
\norm{F(\eta)}_{N\times N}&\lesssim \norm{F(\eta)}_{L^1L^2 \times L^1L^2}\\
&\lesssim \norm{|\eta|^{2\sigma+1}}_{L^1L^2}+\norm{|\eta|\int_{-\infty}^x |\eta|^{2\sigma}dy}_{L^1L^{\infty}}\\
&\lesssim \norm{|\eta|}_{L^{\infty}L^2}\norm{|\eta|}_{L^4L^{\infty}}^4\norm{|\eta|}^{2\sigma-4}_{L^{\infty}L^{\infty}}+\norm{|\eta|}_{L^{\infty}L^2}\norm{|\eta|}^{2\sigma}_{L^{2\sigma}_{t,x}}\\
&\lesssim \norm{\eta}_{S^1\times S^1}^{2\sigma+1}.
\end{align*}
Using Lemma \ref{lm1} again, we have
\[
|\partial_xF(\eta)|\lesssim |(\eta,\partial_x\eta)||\eta|^{2\sigma}+|\partial_x\eta| \left|\int_{-\infty}^x|\eta|^{2\sigma}dy\right|
\]
Thus, similarly as above, we have
\begin{align*}
\norm{\partial_xF(\eta)}_{N\times N}&\lesssim \norm{\partial_xF(\eta)}_{L^1L^2\times L^1L^2}\\
&\lesssim \norm{(|\eta|,|\partial_x\eta|)}_{L^{\infty}L^2}\norm{|\eta|}^4_{L^4L^{\infty}}\norm{|\eta|}_{L^{\infty}L^{\infty}}^{2\sigma-4}+\norm{|\partial_x\eta|}_{L^{\infty}L^2}\norm{|\eta|}_{L^{2\sigma}_{t,x}}^{2\sigma}\\
&\lesssim \norm{\eta}_{S^1\times S^1}^{2\sigma+1}.
\end{align*}
Combining the above, we have
\begin{align*}
\norm{\eta}_{S^1\times S^1}&\leq C\norm{\eta(0)}_{H^1\times H^1}+C\norm{\eta}_{S^1\times S^1}^{2\sigma+1},
\end{align*}
for some constant $C>0$. Thus, for $\norm{\eta(0)}_{H^1\times H^1}$ small enough, by continuous argument, the solution $\eta$ exists globally in time. Moreover, $\norm{\eta}_{S^1(\R) \times S^1(\R)}\lesssim \norm{\eta_0}_{H^1}$ and $F(\eta) \in  N^1(\R) \times N^1(\R)$. By classical argument, there exists $\varphi^{\pm},\psi^{\pm}$ such that 
$$\lim_{t\rightarrow \pm\infty} \norm{\eta - e^{it\partial_x^2}(\varphi^{\pm},\psi^{\pm})}_{H^1\times H^1} \rightarrow 0.$$
The unique existence of $\varphi^{\pm},\psi^{\pm}$ is followed by the uniqueness of $H^2$ solutions of \eqref{eq1} by given initial data (see for instance \cite[Theorem 1.1]{HaOz16}) and the uniqueness of the limit of $e^{-it\partial_x^2}\eta(t)$ as $t\rightarrow\pm\infty$. Hence, the proof is completed.

\bibliographystyle{abbrv}
\bibliography{paper8}

\end{document}